\documentclass{asl}
\usepackage{graphicx}
\usepackage[english]{babel}
\usepackage{color}
\usepackage{hyperref}
\newcommand{\mb}{\mathbf}
\newcommand{\mc}{\mathcal}
\newcommand{\la}{\langle}
\newcommand{\ra}{\rangle}
\newtheorem{thm}{Theorem}[section]
\newtheorem{pro}[thm]{Proposition}
\newtheorem{cor}[thm]{Corollary}
\newtheorem{lem}[thm]{Lemma}
\theoremstyle{definition}
\newtheorem{df}[thm]{Definition}
\newtheorem{rem}[thm]{Remark}
 \newtheorem{con}[thm]{Conjecture}

\begin{document}
	\title{Local initial segments of the Turing degrees}
	\author{Bj\o rn Kjos-Hanssen}

	\address{Mathematisches Institut\\ Ruprecht-Karls-Universit\"at Heidelberg\\
	D-69120 Heidelberg, Germany}

	\email{kjos@math.berkeley.edu}


	\begin{abstract}
	Recent results on initial segments of the Turing degrees are presented, and
	some conjectures about initial segments that have implications for the existence of non-trivial automorphisms of the Turing degrees are indicated.
	\end{abstract}

	\maketitle

	\section{Introduction}
		This article concerns the algebraic study of the upper
		semilattice of Turing degrees. Upper semilattices of interest in this regard tend
		to have a least element, hence for convenience the following definition is made.

		\begin{df}\label{df: 1.1}
			A unital semilattice (usl) is a structure $L = (L, *, e)$ satisfying the following equalities for all $a, b, c \in L$.
			\begin{enumerate}
				\item $a * (b * c) = (a * b) * c$,
				\item $a * b = b * a$,
				\item $a * a = a$, and
				\item $a * e = a$.
			\end{enumerate}
			A bounded unital semilattice (busl) is a structure $L = (L, *, e, z)$ such that
			$(L, *, e)$ is an usl and such that the following equality holds for all $a\in L$:
			\begin{enumerate}
				\item[5.] $a * z = z$.
			\end{enumerate}
		\end{df}

		The Turing degrees $\mc D$ = ($\mc D$, $\cup$, 0) where 0 is the degree of the recursive sets
		and a $\cup$ b is the join (least upper bound) of the degrees a, b is an example of an
		usl. Of course semilattices have natural orderings; that is a $\leq$ b $\leftrightarrow$ a $\cup$ b = b,
		and a $\leq$ b just in case there are sets A, B $\subseteq$ $\omega$ of degree a, b, respectively, such
		that A can be computed using an oracle for B. Note that $\mc D$ has size $2^{\aleph_0}$
		and has the countable predecessor property, i.e. for no $\mb a\in\mc D$ are there uncountably
		many $\mb b < \mb a$. The jump operator $\mb a \mapsto \mb a'$ is included in the language under
		consideration below. Note however that in any ideal where the jump is invariant
		or definable, such as $\mc D$ itself \cite{r20}, jump can be removed from the language.

		An \emph{initial segment} of $\mc D$ is a subset of $\mc D$ which is downward closed. An \emph{ideal}
		of $\mc D$ is a set $\mc I \subseteq \mc D$ such that there exists an usl $\mc L$ and an usl homomorphism
		$\varphi:\mc D\rightarrow \mc L$ such that $\mb a \in\mc I \leftrightarrow \varphi(\mb a) = e$.
		\begin{pro}\label{pro: 1.2}
			Equivalently, $\mc I$ is a nonempty initial segment and closed under join.
		\end{pro}
		\begin{proof}
			In one direction, let $\mc L$ be the set of equivalence classes of the relation
			$E$ defined by $a E b \leftrightarrow \mb a \cup \mb i = \mb b \cup \mb i$ for some $\mb i\in\mc I$,
			and let $*$ be defined by
			$\varphi(\mb a) * \varphi(\mb b) = \varphi(\mb a \cup\mb b)$.
			To check that this is well-defined, suppose $\mb a_k$, $\mb b_k$ are
			in $\mc D$, and $\varphi(\mb a_k) = \varphi(\mb b_k)$, for $k = 0, 1$.
			Then there exists $\mb i_k$ in $\mc I$ such that
			$\mb a_k \cup \mb i_k = \mb b_k \cup \mb i_k$ for $k = 0, 1$. Let $\mb i = \mb i_0 \cup \mb i_1$.
			Then $\mb a_0\cup\mb b_0\cup\mb i = \mb a_1\cup\mb b_1\cup\mb i$, so
			$\varphi(\mb a_0\cup\mb b_0) = \varphi(\mb a_1\cup\mb b_1)$, as desired. In the other direction, $\mc I$ is nonempty since
			it contains $\mb 0$, and if $\mb a\leq\mb b\in\mc I$ then $\mb a\cup\mb b = \mb b$ so
			$\varphi(\mb a) = \varphi(\mb a)*e = \varphi(\mb a)*\varphi(\mb b) = \varphi(\mb a\cup\mb b) = \varphi(\mb b) = e$.
		\end{proof}

		Spector \cite{r22} found the first nontrivial usl isomorphism type of an ideal of $\mc D$;
		namely, an ideal with two elements. This was extended step by step, to all finite
		lattices by Lerman \cite{r10}, to all countable usls by Lachlan and Lebeuf \cite{r9} and to all
		size $\leq\aleph_1$ usls with the countable predecessor property by Abraham and Shore \cite{r1}.
		But Groszek and Slaman \cite{r5} showed that it is consistent with ZFC that the
		same statement with $2^{\aleph_0}$ in place of $\aleph_1$ is false.

		\emph{Local} initial segments are initial segments of the degrees below a given degree,
		for example a degree above $\mb 0''$ or a nonzero r.e.\ degree.

		By an \emph{automorphism} of $\mc D$ is meant an usl automorphism of $\mc D$. Since the
		jump operator is definable in the language $\{\cup\}$ \cite{r20}, every automorphism is jump-preserving.
		Slaman and Woodin \cite{r21} showed that if $\pi$ is an automorphism of $\mc D$
		and $\mb x\geq\mb 0''$ then $\pi(\mb x) = \mb x$.

		Another proof of this fact using Shore's coding with exact pair technique was
		presented in \cite{r15}. Historically, initial segments were used to obtain partial results
		toward the rigidity of the Turing degrees. The work of Slaman and Woodin
		rendered this use obsolete; however recent results may revive this application of
		initial segment theory.

	\section{Results and applications}
		 This section starts with a presentation of a
		new proof of rigidity above $\mb 0''$ that uses the results on ideals of $\mc D$ of \cite{r17}.
		\begin{df}\label{df: 2.1}
			Let $\mb a$ be a Turing degree. An usl $\mc U$ is called $\Sigma^0_1(\mb a)$-presentable
			if there exist countable structures $\mc L = (L, \leq, \lor)$ and $\mc L' = (\omega, \leq' , \lor' )$ such that
			$\mc U$ is isomorphic to $\mc L$ and such that the following conditions hold.
			\begin{enumerate}
				\item $\leq'$ is a transitive and reflexive binary relation on $\omega$ which is r.e.\ in $\mb a$,
				\item $\lor'$ is a binary operation on $\omega$ which is recursive in $\mb a$,
				\item $L$ is the set of equivalence classes of the equivalence relation on $\omega$ given by
				$a\leq' b \,\&\, b \leq' a$,
				\item $\leq$ is the

					relation on $L$ induced by $\leq'$, and
				\item $\lor$ is the operation on

					$L$ induced by $\lor'$.
			\end{enumerate}
		\end{df}

		For $n\geq 0$, $\mc U$ is called $\Sigma^0_{n+1}(\mb a)$-presentable if it is $\Sigma^0_1(\mb a^{(n)})$-presentable, where
		$\mb a^{(n)}$ denotes the $n$th jump of $\mb a$. $\mc U$ is called $\mb a$-presentable if there exist such $\mc L$,
		$\mc L'$ with $\leq'$ recursive in $\mb a$, and $\mc L$ isomorphic to $\mc L'$.

		\begin{lem}\label{lem: 2.2}
			$[\mb a, \mb b]$ is $\Sigma^0_3(\mb b)$-presentable, for any Turing degrees $\mb a\leq\mb b$.
		\end{lem}

		\begin{proof}
			Let $B\in\mb b$, $A\in\mb a$, choose $e\in\omega$ such that $A = \{e\}^B$, the $e$th Turing
			functional applied to $B$, and let
			\[
				C = \{i \mid \{i\}^B\text{ is total and }\{e\}^B\leq_T\{i\}^B \}.
			\]
			The set $C$ is $\Sigma^0_3(B)$ by a standard argument, so

			$C = \{h(n) | n <\omega\}$ for some injective $h\leq_T B''$; define $i\leq' j\leftrightarrow \{h(i)\}^B\leq_T\{h(j)\}^B$.

			If $f_i$ for $i\in\{0, 1\}$ are recursive functions such that every element of $\omega$ is in
			the range of exactly one $f_i$, then a representative of the Turing degree $\mb a\cup\mb b$ is
			given by $A\oplus B = \{f_0(x) \mid x \in A\} \cup \{f_1(x) \mid x \in B\}$. By abuse of notation, define
			a recursive function $\oplus$ such that for $e_0$, $e_1\in\omega$, $X\subseteq\omega$ and $i\in\{0, 1\}$,
			$\{e_0\oplus e_1\}^X(f_i(x)) = \{e_i\}^X(x)$.
			Clearly if the $\{e_i\}^X$ are (characteristic functions of) subsets of $\omega$ then
			$\{e_0\oplus e_1\}^X = \{e_0\}^X\oplus\{e_1\}^X$.
			Now define $i\lor' j = h^{-1}(h(i)\oplus h(j))$ and observe that the conditions of Definition \ref{df: 2.1} obtain.
		\end{proof}

		\begin{proof}[Third proof of rigidity above $\mb 0''$.]
			Suppose $\pi$ is an automorphism of $\mc D$
			and $\mb x\in\mc D$. Then $(\mc D, \mb x) \cong (\mc D, \pi\mb x)$;
			in particular, if one defines $[\mb a, \mb b] = \{\mb y \in\mc D \mid \mb a\leq\mb y\leq\mb b\}$, then $[\mb 0, \mb x]\cong [\mb 0, \pi\mb x]$.
			Also, using the invariance of the jump
			under $\pi$, if one defines $I_{\mb x} = \{$ usl isomorphism types of initial segments $[\mb 0, \mb g]$
			with $\mb g'' \leq x$ $\}$, then it follows that $I_{\mb x} = I_{\pi\mb x}$.

			Now define $S_{\mb x} = \{\text{ $\Sigma_1(x)$-presentable usls }\}$. Assuming that $\mb x \geq \mb 0''$, it is
			shown in \cite{r17} that $I_{\mb x} = S_{\mb x}$. It is also possible to show that for any Turing
			degrees $\mb x$, $\mb y$, if $S_{\mb x} = S_{\mb y}$ then $\mb x = \mb y$ by constructing appropriate usls, along the
			lines of the use of Slaman--Woodin sets in \cite{r15}. Now if $\mb x\geq \mb 0''$ then by invariance
			of the jump also $\pi\mb x\geq\mb 0''$, so $S_{\mb x} = I_{\mb x} = I_{\pi\mb x} = S_{\pi\mb x}$, so $\mb x = \pi\mb x$,
			completing the proof.
		\end{proof}

		The inclusion $I_{\mb x} \subseteq S_{\mb x}$ holds for any degree $\mb x$. Namely, $[\mb 0, \mb g]$ is $\Sigma_1(\mb g'')$-presentable,
		hence since $\mb g''\leq \mb x$ also $\Sigma_1(\mb x)$-presentable.

		For the inclusion $S_{\mb x} \subseteq I_{\mb x}$ one constructs, given a $\Sigma_1(\mb x)$-presentable usl $\mc L$,
		a function $g:\omega\rightarrow\omega$ and functions $g_k:\omega\rightarrow\omega$ for each $k\in\mc L$, such that
		$[\mb 0, \mb g] = \{g_k \mid k \in L\}$ and $L\models k\leq m$ just in case $g_k \leq_T g_m$, and such that
		$\mb g''\leq\mb x$. This is done by breaking the problem up into the following requirements:
		\begin{enumerate}
			\item For each $e$, if $\{e\}^g$
				is total then there exists $k$ such that $\{e\}^g\equiv_T g_k$. This
				is accomplished using an elaboration of Spector's splitting tree technique.
			\item If $k\not\leq m$, then $g_k \not\leq_T g_m$. This is accomplished using a diagonalization
				argument, see, e.g., \cite{r11}.
			\item It can be determined recursively in $\mb x$ and uniformly in $e$ whether or not $\{e\}^g$
				is total. This is accomplished using $e$-total trees, a well-known technique,
				see, e.g., \cite{r11}.
			\item If $k \leq m$ then $g_k \leq_T g_m$.
		\end{enumerate}

		The requirements under item 4 are satisfied using a representation of $\mc L$ as a set
		of equivalence relations, much as in \cite{r9}. The methods of that paper can be used
		to obtain the analogous result for $\mb x$-presentable instead of $\Sigma_1(\mb x)$-presentable
		usls. The main obstacle faced in extending their result was to nd a suitable usl
		representation. Instead of strengthening the already quite complicated representations of \cite{r9},
		a natural representation was found based on a paper of Pudl\'ak \cite{r16},
		a paper that has been celebrated for quite different reasons in the past among lattice theorists.

		The proof sketched above uses the theorem of Slaman--Woodin that all automorphisms preserve the double jump,
		so it does not replace the Slaman--Woodin
		manuscript. Perhaps the proofs and results above are more interesting in relation
		to rigidity above $\mb 0'$, which might be solvable using initial segments.

		Another proof can be given which does not require the new results on initial
		segments but is more roundabout:

		\begin{proof}[Fourth proof of rigidity above $\mb 0''$:] The construction of Lachlan-Lebeuf
			with $e$-total trees and relativization tells us that every $\mb x$-presentable busl is an
			initial segment $[\mb 0, \mb g]$ of the degrees below $\mb x$, for any $\mb x$ above $\mb 0''$, with $\mb g''\leq\mb x$.
			On the other hand, any such initial segment must be $\Sigma_1(\mb x)$-presentable, and if
			$\mb y$ is not below $\mb x$ then one can construct an usl which is $\mb y$-presentable but not
			$\Sigma_1(x)$-presentable. Hence $[\mb 0, \mb x]$ = the set of Turing degrees $\mb z$ such that
			every $\mb z$-presentable busl is an initial segment $[\mb 0, \mb g]$ below $\mb x$ with $\mb g''\leq\mb x$ and the theorem
			follows.
		\end{proof}

		Now as mentioned the result on automorphisms has been proved twice before,
		and two more proofs were presented above. The question is whether any of the
		available proofs shed light on the open problem:
		does there exist an automorphism $\pi$ of $\mc D$ and a degree $\mb x\geq\mb 0'$ such that $\pi\mb x \ne \mb x$?
		In other words, the
		problem asks if $\mb 0''$ can be replaced by $\mb 0'$.
		To pursue this problem using constructions of ideals $[\mb 0, \mb g]$, one would probably need to look at ideals of $[\mb 0, \mb 0']$
		and, by relativization $[\mb d, \mb d']$ for $\mb d\in\mc D$. A first result in that direction is the
		following.
		\begin{thm}\label{thm: 2.3}
			If $\mb x\geq\mb 0''$ or $\mb x = \mb 0'$ and $\mc L$ is a $\Sigma_2(\mb x)$-presentable usl then
			there exists $\mb g<\mb x$ such that $\mc L\cong [\mb 0, \mb g]$.
		\end{thm}
		\begin{proof}[Proof sketch.]
			First suppose $\mb x\geq\mb 0''$. The argument of Lachlan and Lebeuf
			\cite{r9} shows that the theorem holds with $\mb x$-presentable in place of $\Sigma^0_2(\mb x)$-presentable.
			Using a finite injury argument as in \cite[Exercise VIII.1.16]{r11} this result is strengthened to $\mb x'$-presentable. The point is that if the injury is finite then the construction of any tree in a chain $T_0\supseteq T_1\supseteq\dots$ can start over again finitely often, but
			keeping the part of $g$ built so far, so there is no injury to $g$. This can be done
			because it does not matter where on $T_n$ the root of $T_{n+1}$ is put. Finally, the
			use of Pudl\'ak representations, discussed in the next section, allows the strengthening from $\mb x'$-presentable to $\Sigma^0_1(\mb x')$-presentable.
			If $\mb x = \mb 0'$ then we borrow the
			techniques from the construction of Lerman \cite[Chapter XII]{r11} and again use
			Pudl\'ak representations to deal with the high complexity of the usl.
		\end{proof}
		\begin{df}\label{df: 2.4}
			For $n\geq 1$ we say that $\mb g$ is generalized low$_n$, or GL$_n$, if
			$\mb g^{(n)} = (\mb g \cup \mb 0')^{(n-1)}$.
		\end{df}
		\begin{cor}\label{cor: 2.5}
			If $\mb x \geq\mb 0''$ or $\mb x = \mb 0'$ then the following are equivalent for any
			countable busl $L$ which happens to be a lattice:
			\begin{enumerate}
				\item[(i)] $L$ is $\Sigma^0_2(\mb x)$-presentable.
				\item[(ii)] $L$ is isomorphic to $[\mb 0, \mb g]$ for some $\mb g<\mb x$.
			\end{enumerate}
			In particular, the following are equivalent:
			\begin{enumerate}
				\item $L$ is a $\Sigma_3$-presentable bounded lattice.
				\item There exists $\mb g<\mb 0'$ such that $L\cong [\mb 0, \mb g]$ and $[\mb 0, \mb g]$ is a lattice.
			\end{enumerate}
		\end{cor}
		\begin{proof}
			By Jockusch and Posner \cite{r6} if $[\mb 0, \mb g]$ is a lattice then $\mb g$ is GL$_2$, $\mb g'' = (\mb g \cup \mb 0 ')'$, and so $[\mb 0, \mb g]$ is $\Sigma_2(\mb g \cup \mb 0')$-presentable, and since both $\mb g$ and $\mb 0'$ are
			below $\mb x$, also $\Sigma_2(\mb x)$-presentable.
		\end{proof}
		The first partial results toward Corollary 2.5 were obtained by Spector \cite{r22}
		and Sacks \cite{r18} who found minimal degrees below $\mb 0''$ and $\mb 0'$, respectively. Next,
		it is natural to consider initial segments below arbitrary nonzero r.e. degrees\ or
		low initial segments.
		\begin{thm}\label{thm: 2.6}
			If $\mb a\in\mc D$ and $L$ is a $\Sigma^0_2(\mb a)$-presentable busl then $L\cong [\mb a, \mb g]$ for some $\mb g$ satisfying $\mb g' = \mb a'$.
		\end{thm}
		\begin{proof}[Proof sketch]
			Lerman (\cite{r11} Corollary XII.5.10) obtains the analogous result for L being $\mb a'$-presentable rather than $\Sigma^0_1(\mb a')$-presentable,
			but using Pudl\'ak representations there is no essential distinction.
		\end{proof}
		\begin{con}[Lowness conjecture]\label{con: 2.7}
			If $\mb a\in\mc D$, $\mb x\geq\mb a''$ or $\mb x = \mb a'$ and
			$\mb g\leq\mb x$, if $[\mb a, \mb g]$ is a lattice
			and $\mb g\in$ GL$_1(\mb a)$, i.e., $\mb g' = \mb g\cup\mb a'$, then $[\mb a, \mb g]$ is $\Sigma^0_1(\mb x)$-presentable.
		\end{con}
		A stronger version of the conjecture:
			If $\mb a\in\mc D$, $\mb x\geq\mb a''$ or $\mb x = \mb a'$ and
			$\mb g\leq\mb x$, if $[\mb a, \mb g]$ contains no 1-generic relative to $\mb a$
			and $\mb g\in$ GL$_1(\mb a)$, i.e., $\mb g = \mb g \cup\mb a'$, then $[\mb a, \mb g]$ is $\Sigma^0_1(\mb x)$-presentable.

		\begin{lem}\label{lem: 2.8}
			If Conjecture \ref{con: 2.7} is true then every automorphism of $\mc D$ is the
			identity above $\mb 0'$.
		\end{lem}
		\begin{proof}
			Let $\mb x\geq\mb 0'$ and by the Friedberg jump inversion theorem let $\mb a$ be such
			that $\mb x = \mb a'$. Consider the usl isomorphism types of intervals of the form $[\mb a, \mb g]$
			with $\mb g' = \mb a'$ that are lattices. By Theorem \ref{thm: 2.6} and Conjecture \ref{con: 2.7}, these are
			exactly the $\Sigma^0_1(x)$-presentable busls, and this class determines $\mb x$ uniquely.
		\end{proof}

		The usual constructions of initial segments below $\mb 0''$ using forcing with recursive perfect trees yield degrees $\mb g$ that are hyperimmune-free (HIF).
		This means that every function recursive in $\mb g$ is dominated by a recursive function.
		In general, $\mb g$ is HIF relative to $\mb a$ if every function recursive in $\mb g\cup\mb a$ is dominated by
		a function recursive in $\mb a$.

		\begin{lem}\label{lem: 2.9}
			Conjecture \ref{con: 2.7} is true for any g which is HIF relative to a.
		\end{lem}
		\begin{proof}
			It can be readily shown that HIF degrees $\mb g$ satisfy $\mb g'' = \mb g' \cup \mb 0''$. If a
			degree $\mb g$ below $\mb x \geq\mb 0''$ is both HIF and GL$_1$, then
			$\mb g'' = \mb g' \cup\mb 0'' = (\mb g\cup\mb 0')\cup\mb 0'' = \mb g\cup\mb 0'' \leq \mb x$, so
			$[\mb 0, \mb g]$ is $\Sigma^0_1(\mb x)$-presentable, and this argument relativizes to $\mb a\in\mc D$.
			And for any $\mb a\in\mc D$, there are no HIF degrees in the interval $(\mb a, \mb a']$ (see \cite{r14}), so
			the lemma is trivial for $\mb x = \mb a'$.
		\end{proof}

		On the other hand, there are HIF degrees $\mb g$ below $\mb x$ with $[\mb 0, \mb g]$ of any $\Sigma^0_2(\mb x)$-
		presentable isomorphism type; in fact these degrees can all be forced to be not
		GL$_1$.

		It follows from results of Shore \cite{r19} that if $\mb g$ is not GL$_2$ then the lemma is
		best possible, i.e. every presentation of $[\mb 0, \mb g]$ has degree at least $\mb g^{(3)}$.
		On the other hand initial segment constructions lead to GL$_2$ degrees $\mb g$, and such initial
		segments can have arbitrary prescribed isomorphism type, hence the interval
		$[\mb 0, \mb g]$ may even have a presentation of degree $\mb 0$. The fact that non-GL$_2$ initial
		segments $[\mb 0, \mb x]$ are complicated may be seen as evidence against Conjecture \ref{con: 2.7}.
		However Shore's proof uses explicitly the fact that $[\mb 0, \mb x]$ is not a lattice. Namely,
		the proof involves a coding of the set $X^{(3)}$, (where $X$ is a set of degree $\mb x$) by
		\[
			n \in X^{(3)}\leftrightarrow \mb g_n \leq \mb y_0 \,\&\, \mb g_n \leq \mb y_1
		\]
		for certain degrees $\mb g_0, \mb g_1, \dots, \mb y_0, \mb y_1< \mb x$, and the meet (greatest lower bound)
		$\mb y_0\cap \mb y_1$ can be shown to not exist.

	\section{An overview of the proofs}
		Considering the theorem that every $\Sigma_1(x)$-
		presentable busl is isomorphic to $[\mb 0, \mb g]$ for $\mb g<\mb x$ if $\mb x\geq \mb 0''$,
		one has the Turing
		degrees $\mc D$ on one hand, and a countable usl $L$ on the other, and it is desired
		to embed $L$ into $\mc D$ as an initial segment. The first step is to break $L$ up into
		countably many finite pieces. One way to do this would be to write $L = \cup_iL_i$
		where each $L_i$ is a finite busl substructure of $L$. However it is preferable to use
		a more general representation of $L$ as a direct limit. That is, find finite busl
		substructures $L_i$ of $L$ and busl homomorphisms $\varphi_i:L_i\rightarrow L_{i+1}$ such that $L$ is
		isomorphic to the disjoint union of the $L_i$ modulo an equivalence relation $\approx$ generated by the relations
		$a \approx\varphi_i(a)$ for each $i<\omega$ and each $a\in L_i$. The advantage
		of using homomorphisms rather than embeddings (injective homomorphisms) is
		that if $L$ is $\Sigma_1(\mb x)$-presentable then a suitable sequence of $L_i$ and $\varphi_i$ for
		$i<\omega$
		can be found recursively in $\mb x$.

		To each $L_i$ associate a countably infinite chain
		$\Theta_0(L_i)\subseteq\Theta_1(L_i)\subseteq\dots$ of
		finite sets, where both the sequence and its union are referred to as $\Theta(L_i)$ or
		just $\Theta$ when no confusion is likely; note that this gives a doubly infinite array
		$\Theta^i_j = \Theta_j(L_i)$, $i, j<\omega$. On each $\Theta$j there are finite functions
		$f_k = f^{i,j}_k : \Theta_j\rightarrow\Theta_j$
		for each $k\in L_i$, such that $f^{i,j+1}_k$ extends $f^{i,j}_k$ as a function.
		This is done in such
		a way that the associated equivalence relations given by
		$x\equiv_k y\leftrightarrow f_k(x) = f_k(y)$ have the property that $L \models k \leq m$ just in case $\equiv_k\supseteq\equiv_m$
		(considering an
		equivalence relation as a set of pairs).

		In the end a function $g:\omega\rightarrow\omega$ (with degree denoted in boldface, $\mb g$) is needed
		such that the interval $[\mb 0, \mb g]$ in the Turing degrees is equal to $\{\mb g_k \mid k\in L\}$, where
		$g_k(x) = f_k(g(x))$. The way the functions $f_k$ are set up, if $k\leq m$ in $L$ then
		$g_k(x) = f_k(g_m(x))$. Hence in this case $g_k\leq_T g_m$, in fact $g_k\leq_{tt} g_m$. Note
		however that the reduction procedure using $f_k$ is not a many-one reduction; this
		is no accident, as the many-one degrees form a distributive usl, and usls are not
		in general distributive.

		Some version of the splitting tree technique of Spector seems to be a necessary
		tool for constructing initial segments of the Turing degrees. Hence $g$ should lie
		on trees $T_0\supseteq T_1\supseteq\dots$. These trees will be \emph{trees for $L_i$} for various $i$, for example
		$T_{2i}$ and $T_{2i+1}$ can be trees for $L_i$ for each $i<\omega$. The domain of a tree for $L_i$
		will be
		\[
			S(\Theta(L_i)) = \{\sigma \in \omega^{<\omega} \mid \forall x < |\sigma| (\sigma(x) \in \Theta_x(L_i))\}
		\]
		and the range will be contained in $S(\Theta(L_0))$. In order to have $T_{i+1}$ be a subtree
		of $T_i$, $\Theta(L_{i+1})$ must be embedded into $\Theta(L_i)$ in a certain sense. Because $L_{i+1}$
		may have many more elements than $L_i$, it is necessary to use an increasing
		function $m_i$ and embed $\Theta_j(L_{i+1})$ into $\Theta_{m_i(j)}(L_i)$.

		The trees $T_{2i+1}$ can be used to satisfy the requirements (1),(2),(3) above. The
		trees $T_{2i+2}$ serve only the purpose of going from $L_i$ to $L_{i+1}$. For simplicity
		assume that only the trees with even subscripts are needed, so $T_{2i+2}$ is referred
		to as ``$T_{i+1}$'' in the following. First, the root $T_{i+1}(\emptyset)$ is taken to be any $T_i(\sigma$)
		with $|\sigma|>m_i(0)$. The reason is that then $T_{i+1}(\la x\ra)$ for $x\in\Theta_0(L_{i+1})$ can be
		defined as a string extending $T_i(\sigma*\la x\ra)$. The embedding of $\Theta(L_{i+1})$ into $\Theta(L_i)$
		can be considered to be an inclusion map. If $T_{i+1}(\la x\ra) = T_i(\sigma*\la x\ra^{m_i(1)-m_i(0)})$
		then the process can be continued in the same way at the next level of $T_{i+1}$.

		In the construction of initial segments of the Turing degrees a notion of homogeneity arises
		(see for example Lerman \cite{r11}). With assistance from Pavel
		Pudl\'ak and Ralph McKenzie, a lattice theoretic way to understand this notion
		was discovered. The most natural notion of homogeneity is slightly different
		and weaker than the notion of weak homogeneity in Lerman's book, but is still
		sufficient for the needs of initial segment constructions once a small addition to
		such constructions is made. The lattice theoretical meaning of this new notion
		is that a set of equivalence relations which form a lattice under inclusion has the
		homogeneity property just in case the set of equivalence relations is the congruence lattice of an algebra
		(for definitions and properties of congruence lattices
		see \cite{r3}), a fundamental concept of universal algebra.

		The new notion of homogeneity is the following.
		\begin{df}\label{df: 3.1}
			$\Theta$ is called \emph{homogeneous} if the following condition obtains.
			Suppose for each $k\in L$, $a\equiv_k b\Rightarrow c\equiv_k d$.
			Then there exist $z_1,\dots, z_n$ and
			maps $f_i:\Theta\rightarrow\Theta$ preserving all the $\equiv_k$, such that
			\begin{enumerate}
				\item $\{f_0(a), f_0(b)\} = \{c, z_1\}$,
				\item $\{f_i(a), f_i(b)\} = \{z_i, z_{i+1}\}$, and
				\item $\{f_n(a), f_n(b)\} = \{z_n, d\}$.
			\end{enumerate}
		\end{df}

		Note that only equality of sets is required here, as opposed to equality of
		ordered pairs.

		\begin{df}\label{df: 3.2}
			If $A$ is a set, $f$ is a binary operation on $A$ and $E$ is an
			equivalence relation on $A$, then $E$ is a \emph{congruence relation} if whenever $xEx'$ and
			$yEy'$ then $f(x, y)Ef(x' , y' )$.
		\end{df}

		The set of all congruence relations on a given set may be ordered by inclusion
		if an equivalence relation is considered as a set of pairs. The resulting partial
		order is always a lattice. A classic example: the lattice of normal subgroups of
		a group $G$ is isomorphic to the congruence lattice of $G$. The isomorphism sends
		the subgroup $H$ to the relation $xEy\Leftrightarrow xy^{-1}\in H$.

		The following observation may be credited to Mal'cev \cite{r12}, \cite{r13}.
		\begin{lem}\label{lem: 3.3}
			The congruence lattice of an algebra is homogeneous.
		\end{lem}
		\begin{proof}[Proof sketch]
			Suppose $a, b, c, d\in\Theta$ and $a\equiv_k b \rightarrow c\equiv_k d$ for each $k\in L$.
			Then $\la c, d\ra$ is in the equivalence relation $E$ generated by $\la a, b\ra$ within $\Theta$. But
			since $\Theta$ is a congruence lattice, $E$ is identical to the congruence relation on $\Theta$
			generated by $\la a, b\ra$. So $\la c, d\ra$ is in the equivalence relation generated by pairs
			$\la f(a), f(b)\ra$ for $f$ homomorphism of $\Theta$.
		\end{proof}
		\begin{rem}\label{rem: 3.4}
			The following is an open problem in lattice theory which has
			a bearing on initial segment constructions. Call a finite structure in a finite
			language with no relation symbols a finite algebra. Say that a finite lattice is
			CLFA if it is isomorphic to the congruence lattice of some finite algebra. Is
			every finite lattice CLFA? In fact, given $n\in\omega$, let $M_n$ be the lattice which
			has an antichain of size $n$, a top and a bottom. It is not known whether or
			not every $M_n$ is CLFA. The class CLFA includes the finite lattices that have a
			homogeneous lattice table in the sense of Lerman \cite{r11}. Lachlan \cite{r8} makes a claim
			that is equivalent to saying that every CLFA is an ideal in $\mc D$, but does not give
			the construction as Lerman \cite{r10} had already shown that every finite lattice is an
			ideal in $\mc D$, using an infinite representation $\Theta$. With finite $\Theta$, it seems CLFA is
			all one can get. Whether $\Theta$ is finite or infinite corresponds to whether the trees
			have branching rates bounded by a constant, or a by a recursive function.
		\end{rem}

		Once it is realized that the congruence lattice of an algebra is homogeneous,
		it is natural to consider Pudl\'ak's paper \cite{r16} where a representation of a finite
		(or more generally, an algebraic) lattice as a congruence lattice is realized as a
		union of a chain $\Theta_0\subseteq\Theta_1\subseteq\dots$. The question occured to us whether Pudl\'ak's
		construction also had the property that $\Theta(L_{i+1})$ could be embedded into $\Theta(L_i)$
		in an appropriate sense. A slight modication of Pudl\'ak's construction turned
		out to have this property. The resulting representation is signicantly simpler
		than the one used by Lachlan and Lebeuf and in subsequent work. Instead of
		building all the properties needed of the representation into it by brute force, a
		representation is used which was created for a different and simpler purpose but
		which has these properties naturally. For an analogy, one can imagine a natural
		solution to Post's problem compared to the Friedberg-Muchnik solution.
		\begin{figure}\label{Figure 1}
			\setlength{\unitlength}{4144sp}%
\begingroup\makeatletter\ifx\SetFigFont\undefined%
\gdef\SetFigFont#1#2#3#4#5{%
  \reset@font\fontsize{#1}{#2pt}%
  \fontfamily{#3}\fontseries{#4}\fontshape{#5}%
  \selectfont}%
\fi\endgroup%
\begin{picture}(2825,2599)(113,-848)
\put(901,389){\makebox(0,0)[lb]{\smash{\SetFigFont{12}{14.4}{\rmdefault}{\mddefault}{\updefault}\special{ps: gsave 0 0 0 setrgbcolor}$b$\special{ps: grestore}}}}
\thinlines
\special{ps: gsave 0 0 0 setrgbcolor}\put(1576,1064){\line(-1, 2){225}}
\put(1351,1514){\line( 2, 1){450}}
\put(1801,1739){\line( 2,-1){450}}
\put(2251,1514){\line(-1,-2){225}}
\special{ps: grestore}\special{ps: gsave 0 0 0 setrgbcolor}\put(451,164){\line(-1, 2){225}}
\put(226,614){\line( 2, 1){450}}
\put(676,839){\line( 2,-1){450}}
\put(1126,614){\line(-1,-2){225}}
\special{ps: grestore}\special{ps: gsave 0 0 0 setrgbcolor}\put(1351,-61){\line(-1, 2){225}}
\put(1126,389){\line( 2, 1){450}}
\put(1576,614){\line( 2,-1){450}}
\put(2026,389){\line(-1,-2){225}}
\special{ps: grestore}\special{ps: gsave 0 0 0 setrgbcolor}\put(2251,164){\line(-1, 2){225}}
\put(2026,614){\line( 2, 1){450}}
\put(2476,839){\line( 2,-1){450}}
\put(2926,614){\line(-1,-2){225}}
\special{ps: grestore}\special{ps: gsave 0 0 0 setrgbcolor}\put(901,-736){\line( 1, 0){450}}
\special{ps: grestore}\special{ps: gsave 0 0 0 setrgbcolor}\put(676, 52){\vector(-2, 3){  0}}
\put(676, 52){\vector( 2,-3){450}}
\special{ps: grestore}\special{ps: gsave 0 0 0 setrgbcolor}\put(1576,-173){\vector( 3, 4){  0}}
\put(1576,-173){\vector(-3,-4){337.680}}
\special{ps: grestore}\special{ps: gsave 0 0 0 setrgbcolor}\put(2476, 52){\vector( 3, 2){  0}}
\put(2476, 52){\vector(-3,-2){1012.846}}
\special{ps: grestore}\put(1576,1289){\makebox(0,0)[lb]{\smash{\SetFigFont{12}{14.4}{\rmdefault}{\mddefault}{\updefault}\special{ps: gsave 0 0 0 setrgbcolor}$\Theta_2$\special{ps: grestore}}}}
\put(451,-736){\makebox(0,0)[lb]{\smash{\SetFigFont{12}{14.4}{\rmdefault}{\mddefault}{\updefault}\special{ps: gsave 0 0 0 setrgbcolor}$\Theta_0$\special{ps: grestore}}}}
\put(1463,277){\makebox(0,0)[lb]{\smash{\SetFigFont{12}{14.4}{\rmdefault}{\mddefault}{\updefault}\special{ps: gsave 0 0 0 setrgbcolor}$\Theta_1$\special{ps: grestore}}}}
\put(113,389){\makebox(0,0)[lb]{\smash{\SetFigFont{12}{14.4}{\rmdefault}{\mddefault}{\updefault}\special{ps: gsave 0 0 0 setrgbcolor}$a$\special{ps: grestore}}}}
\put(338,727){\makebox(0,0)[lb]{\smash{\SetFigFont{12}{14.4}{\rmdefault}{\mddefault}{\updefault}\special{ps: gsave 0 0 0 setrgbcolor}$b$\special{ps: grestore}}}}
\put(901,727){\makebox(0,0)[lb]{\smash{\SetFigFont{12}{14.4}{\rmdefault}{\mddefault}{\updefault}\special{ps: gsave 0 0 0 setrgbcolor}$a$\special{ps: grestore}}}}
\put(1126,-848){\makebox(0,0)[lb]{\smash{\SetFigFont{12}{14.4}{\rmdefault}{\mddefault}{\updefault}\special{ps: gsave 0 0 0 setrgbcolor}$c$\special{ps: grestore}}}}
\special{ps: gsave 0 0 0 setrgbcolor}\put(1801,1064){\vector(-1, 1){  0}}
\put(1801,1064){\vector( 1,-1){337}}
\special{ps: grestore}\end{picture}
			\caption{Pudl\'ak's construction.}
		\end{figure}
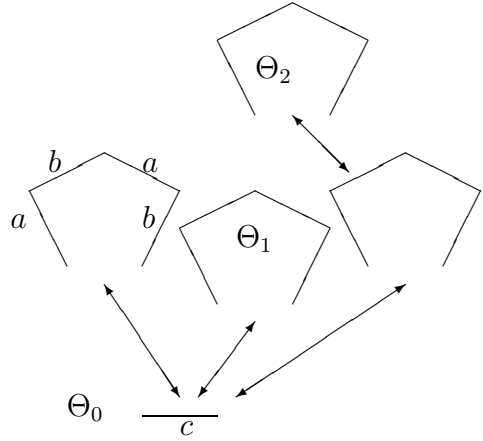

		Pudl\'ak's construction of a representation $\Theta(L)$, where $L$ is now a finite lattice,
		is that of a colored graph, where the colors are elements of $L$. In $\Theta_0$ there is
		just a single edge labelled by an element of $L$ other than 0. Then inductively, in
		$\Theta_{j+1}$, for each edge of $\Theta_j$ labelled $a\in L$ there are four additional edges glued
		together so as to form a pentagon with edges labelled $a, b, c, b, c$, for any pair
		$b, c$ such that $a\cup b\geq c$, as illustrated in Figure \ref{Figure 1}. Equivalence relations $\equiv_m$ for
		$m\in L$ are defined by: $x\equiv_m y$ if there is a path from $x$ to $y$ all of whose edges
		have colors $\leq m$.

		To embed $\Theta(L_{i+1})$ into $\Theta(L_i)$, an apparent problem might be that the natural
		$\varphi_i$ goes from $L_i$ to $L_{i+1}$ which is the opposite direction. There is a natural way to
		embed $\Theta(L_i)$ into $\Theta(L_{i+1})$; simply change the color of any edge labelled $a\in L_i$
		to $\varphi_i(a)$. However, this embedding corresponds to the equivalence $m \leq k$
		just in case $\equiv_m\subseteq\equiv_k$ for the Pudl\'ak representation $\Theta$ and not
		just in case $\equiv_m\supseteq\equiv_k$.
		Instead the dual $L_i^*$ of $L_i$ should be considered, and $\Theta(L_i^*)$ used as representation
		of $L_i$. To embed $\Theta(L_{i+1}^*)$ into $\Theta(L_i^*)$, an usl homomorphism from $L_{i+1}^*$ to $L_i^*$ is
		needed, i.e. a map from $L_{i+1}$ to $L$ preserving meet $\cap$ and greatest element 1.
		Fortunately the following lemma is available.
		\begin{lem}\label{lem: 3.5}
			The category of finite usls $L$ with usl homomorphisms $\varphi$ is self-dual under the contravariant functor
			taking $L$ to $L^*$ and  $\varphi$ to $\varphi^*$, where $\varphi^*$ is the
			\emph{Galois adjoint} \cite{r2} of $\varphi$ given by $\varphi(a)\leq x\leftrightarrow a \leq\varphi^*(x)$.
		\end{lem}
		The following is a partial reconstruction of the history of Pudl\'ak's construction.
		Whitman \cite{r23} showed every lattice embeds in a partition lattice.
		J\'onsson \cite{r7} simplied Whitman's construction.
		Gr\"atzer and Schmidt \cite{r4} characterized congruence lattices of algebras.
		Pudl\'ak proved that Gr\"atzer and Schmidt's proof could be simplied using a slight modication of J\'onsson's construction.
		%\footnote{
		%	The present file has been generated from
		%	\href{http://math.hawaii.edu/~bjoern/Publications/BSL-Kjos-Hanssen.pdf}{a more accurate pdf file (click here)}
		%	using InftyReader and AbiWord.
		%}
	\bibliographystyle{asl}
	\bibliography{BSL-Kjos-Hanssen-abiword}
\end{document}